\documentclass[a4paper,twoside]{amsart}
\usepackage[utf8]{inputenc}
\usepackage{amsfonts}
\usepackage{graphicx}
\newcommand{\overbar}[1]{\mkern 1.5mu\overline{\mkern-1.5mu#1\mkern-1.5mu}\mkern 1.5mu}
\usepackage{enumitem}
\usepackage[hidelinks]{hyperref}

\theoremstyle{plain}
\newtheorem{theorem}{Theorem}[section]

\newtheorem{observation}[theorem]{Observation}
\newtheorem{lemma}[theorem]{Lemma}

\newtheorem{fact}[theorem]{Fact}

\theoremstyle{definition}
\newtheorem{definition}[theorem]{Definition}

\newtheorem*{remark*}{Remark}

\newtheorem*{claim*}{Claim}
\theoremstyle{remark}

\def\Forb{\mathop{\mathrm{Forb}}\nolimits}

\def\str#1{\mathbf {#1}}

\begin{document}
\bibliographystyle{alpha}

\title{Forbidden cycles in metrically homogeneous graphs}

\author[J. Hubi\v cka]{Jan Hubi\v cka}
\address{Department of Applied Mathematics (KAM), Charles University, Ma\-lo\-stransk\'e n\'a\-m\v es\-t\'\i~25, Praha~1, Czech Republic}
\email{hubicka@kam.mff.cuni.cz}

\author[M. Kompatscher]{Michael Kompatscher}
\address{Department of Algebra, Charles University, Prague, Czech Republic}
\email{kompatscher@karlin.mff.cuni.cz}

\author[M. Kone\v cn\'y]{Mat\v ej Kone\v cn\'y}
\address{Institute of Algebra, TU Dresden, Dresden, Germany, and\\ Department of Applied Mathematics (KAM), Charles University, Ma\-lo\-stransk\'e n\'a\-m\v es\-t\'\i~25, Praha~1, Czech Republic}
\email{matej.konecny@tu-dresden.de}

\thanks{Matěj Konečný was supported by a SFG of the Faculty of Mathematics and Physics, Charles University. Michael Kompatscher was supported by Charles University Research Centre programs PRIMUS/SCI/12 and UNCE/SCI/022. In the final stages of this project, Jan Hubička and Matěj Konečný were supported by the project 21-10775S of the Czech Science Foundation (GAČR) and by a project that has received funding from the European Research Council (ERC) under the European Union’s Horizon 2020 research and innovation programme (grant agreement No 810115, DYNASNET), and Michael Kompatscher was supported by a project that has received funding from the European Research Council (ERC) under the European Unions Horizon 2020 research and innovation programme (grant agreement No. 771005, CoCoSym).
}

\begin{abstract}
In a recent paper by a superset of the authors it was proved that for every primitive 3-constrained space $\Gamma$ of finite diameter $\delta$ from Cherlin's catalogue of metrically homogeneous graphs, there exists a finite family $\mathcal F$ of $\{1,\ldots, \delta\}$-edge-labelled cycles such that a $\{1,\ldots, \delta\}$-edge-labelled graph is a subgraph of $\Gamma$ if and only if it contains no homomorphic images of cycles from $\mathcal F$. However, the cycles in the families $\mathcal F$ were not described explicitly as it was not necessary for the analysis of Ramsey expansions and the extension property for partial automorphisms.

This paper fills this gap by providing an explicit description of the cycles in the families $\mathcal F$, heavily using the previous result in the process. Additionally, we explore the potential applications of this result, such as interpreting the graphs as semigroup-valued metric spaces or homogenizations of $\omega$-categorical $\{1,\delta\}$-edge-labelled graphs.


\end{abstract}

\maketitle

\section{Introduction}
A {\em metrically homogeneous graph} is a (countable) connected graph with the
property that the associated metric space is homogeneous.  (Here the {\em
associated metric space} of a graph shares its vertex set, and the
distance between two vertices is the length of the shortest path connecting them.
A metric space is {\em homogeneous} if every isomorphism, or isometry, between finite
subspaces extends to an automorphism of the whole metric space.) In a recent monograph, Cherlin
\cite{Cherlin2013,Cherlin2011b} gives a catalogue of metrically homogeneous graphs which
is conjectured to be complete and confirmed up to diameter
three~\cite{Amato2016}. He describes them by giving a list of forbidden triangles in the associated (path-)metric spaces. This is so far the most elaborate addition to
the classification programme of homogeneous structures. 

In this paper we give an alternative description of metrically homogeneous graphs
by means of forbidden cycles. This is motivated by the applications in the structural
Ramsey theory and topological dynamics outlined in
Section~\ref{sec:motivation}, but the result is of independent combinatorial interest.
We focus on those metrically homogeneous graphs which can be described by means of forbidden
triangles in the associated metric spaces (i.e.~\emph{$3$-constrained} cases) with
a primitive automorphism group. (Recall that an automorphism group is \emph{primitive} if it acts transitively and preserves no nontrivial partition of vertices.)  Cherlin described such graphs by means of five
numeric parameters (see Section~\ref{sec:3-constrained}) which play a key role even in the rest of the catalogue. Thus our families of forbidden cycles can be generalised to the rest of
the catalogue by techniques discussed in greater detail in~\cite{Cherlin2013} as well as in \cite{Aranda2017c}.

Our main result is a precise characterisation of forbidden sub-cycles of every
metric space associated to a primitive 3-constrained metrically homogeneous
graph in the catalogue. We show that in addition to non-metric cycles (i.e. cycles where one edge is greater distance than the sum of rest edges) all the cases can be described as a combination of four naturally defined families as stated in Theorem~\ref{thm:main}. This extends (and
completes) earlier results~\cite{Aranda2017,Aranda2017a,Aranda2017c,Coulson,Konecny2018bc} which prove that
these structures are described by a finite set of forbidden cycles.

Before stating our main result we take time to review the history of the problem
and give some basic definitions.

\subsection{Motivation}
\label{sec:motivation}
Our original motivation stems from the study of Ramsey classes and Ramsey
structures. We refer the reader to one of the recent surveys~\cite{NVT14,Bodirsky2015,hubicka2025twenty}
for precise definitions and we only review some recent developments which led to
our work.

In~\cite{Hubicka2016} it is shown that describing homogeneous structures by
means of forbidden substructures (or, more precisely, obtaining an upper bound
on the size of minimal such substructures) is the key to obtaining a stronger property---the
existence of a Ramsey expansion~\cite[Theorem 2.1]{Hubicka2016}.  This builds on
the ideas of Ne\v se\v ril's earlier result on the existence of a Ramsey expansion of
the Urysohn metric space~\cite{Nevsetvril2007} (a related result was also obtained by Dellamonica and R{\"o}dl~\cite{Dellamonica2012})
which itself extends the earlier technique called \emph{partite construction} which was developed in
1980's and led to a simpler proof of the well known Ne\v set\v ril--R{\"o}dl
theorem~\cite{Nevsetvril1989} (see
\cite{Abramson1978,Nevsetvril1977,Nevsetvril1977b} for original proofs).

Combining the model-theoretic concepts of strong amalgamation,
forbidden substructures and the combinatorial tool of partite construction
resulted in a systematic framework which is used to prove that a given class is
Ramsey. It covers many known examples of Ramsey classes and also gives new
ones~\cite{Hubicka2016}. However, upon finishing these tools it was not clear
which homogeneous structures have such a description.  It is
clear that some structures, such as equivalences and partial orders, can not be
described this way, but there are additional tools making it possible to fit them
to the framework of~\cite{Hubicka2016} and hence find their Ramsey
expansions. Thus it was not clear which examples of homogeneous structures do
not have a good Ramsey expansion.

During an initial discussion of this problem, Cherlin suggested a particular
example in his catalogue of metrically homogeneous graphs as a possible example
of homogeneous structure in finite language which can not be characterised by a
finite family of forbidden substructures. While this example was later shown
to have such a finite description, it was necessary to develop new tools to analyse it.
More recently, the question whether every $\omega$-categorical homogeneous structure has
a ``good'' (\emph{precompact}) Ramsey expansion was answered negatively
in~\cite{Evans2}. However, it still remains open for homogeneous structures in a finite language.

Independently, Ramsey expansions of restricted metric spaces were also
systematically studied by Nguyen Van Th\'e~\cite{The2010}. He has shown the
existence of Ramsey expansions for classes of $S$-metric spaces (i.e. metric
spaces where all distances are in a fixed set $S$) for $|S|\leq 4$. These
results were extended to all meaningful choices of $S$ in~\cite{Hubicka2016}.
Nguyen Van Th\'e was motivated by a long-standing open problem asking whether the class of all finite 
affinely independent Euclidean metric spaces has a precompact Ramsey
expansion~\cite{NVT14,Kechris2005} which seems to be still out of reach of the existing
techniques.

Special metric spaces thus clearly presented and present interesting and challenging
examples in the area. The full analysis of Cherlin's catalogue was started during the
Ramsey DocCourse in Prague in 2016 and was completed a year later~\cite{Aranda2017}.
During this work, new connections were discovered. In particular, essentially
the same techniques can be also used to show the extension property for
partial automorphisms (EPPA) using the Herwig--Lascar Theorem~\cite{herwig2000} (or, more conveniently, its recent strengthening by Hubička, Konečný, and Nešetřil~\cite{Hubicka2018EPPA} which was published between the submission and publication of this paper).
They are also closely related to the stationary independence relation,
which was used by Tent and Ziegler~\cite{Tent2013} to show several properties of the
automorphism group of the Urysohn metric space.\footnote{Between the submission and publication of this paper, Evans, Hubička, Konečný, Li, and Ziegler~\cite{Evanssimplicity} adapted the Tent--Ziegler method for certain generalised metric spaces. In particular, combining with the results of this paper, one can show that the automorphism groups of the structures studied in this paper are simple.}
We shall remark that EPPA for metric spaces was independently shown by Solecki~\cite{solecki2005}
and Vershik~\cite{vershik2008} and generalised by Conant~\cite{Conant2015}, see also~\cite{Hubicka2017sauerconnant,Hubicka2018metricEPPA}.

The analysis carried in~\cite{Aranda2017} does not give a precise description of the
forbidden substructures, only a rather generous upper bound on their size. We hope that having
a precise description will shed more light onto the nature of the catalogue
and also relate it to the concept of
homogenization~\cite{Covington1990,Hubicka2013,Hubicka2009}.  It is not
difficult to see (using the results of~\cite{Aranda2017}) that every metric space associated to a metrically
homogeneous graph of finite diameter $\delta$ can be seen as a homogenization
of a structure containing only distances 1 and $\delta$. This, in turn, can
explain some phenomena, such as twisted automorphisms~\cite{coulson2018twists}
as sketched in Section~\ref{sec:1delta}.

\subsection{The primitive 3-constrained metrically homogeneous graphs}
\label{sec:3-constrained}
In this paper, we shall only be concerned with a subset of the metrically
homogeneous graphs, namely the primitive 3-constrained classes. The following
definition and theorem of Cherlin are simplified to only contain the classes
relevant for our paper.

\label{sec:3constrained}
\begin{definition}[Triangle constraints]
\label{defn:numerical}
Given integers $\delta$, $K_1$, $K_2$, $C_0$ and $C_1$, we consider the class
$\mathcal A^\delta_{K_1,K_2,C_0,C_1}$ of all finite metric spaces $\str{M}=(M,d)$ with integer
distances such that $d(u,v)\leq \delta$ (we call $\delta$ the {\em diameter} of $\mathcal A^\delta_{K_1,K_2,C_0,C_1}$) for every $u,v\in M$
and for every triangle $u,v,w\in M$ with perimeter $p=d(u,v)+d(u,w)+d(v,w)$ the following are true: 
\begin{itemize}
 \setlength\itemsep{0em}
 \item if $p$ is odd then $2K_1 < p < 2K_2 + 2m$,
 \item if $p$ is odd then $p<C_1$, and
 \item if $p$ is even then $p<C_0$.
\end{itemize}
Here $m=\min\{d(u,v),\allowbreak d(u,w),\allowbreak d(v,w)\}$ is the length of the shortest edge of $u,v,w$.
\end{definition}
Intuitively, the parameter $K_1$ forbids all odd cycles shorter than $2K_1+1$, while $K_2$ ensures that the difference in length between even- and odd-distance paths connecting any
pair of vertices is less than $2K_2+1$. The parameters $C_0$ and $C_1$
forbid induced long even and odd cycles respectively. Not every combination of
numerical parameters makes sense and leads to an amalgamation class. Those that do make sense are
described by the definition below and those that lead to an amalgamation class are
characterised by Cherlin's Admissibility Theorem (stated here in simplified form considering only primitive graphs as Theorem~\ref{thm:admissible}).
\begin{definition}[Acceptable numerical parameters]
A sequence of parameters $(\delta,\allowbreak K_1,\allowbreak K_2,\allowbreak C_0,\allowbreak C_1)$ is {\em acceptable} if it satisfies the following conditions:
\begin{itemize}
  \setlength\itemsep{0em}
  \item $3\leq \delta< \infty$;
  \item $1\leq K_1\leq K_2\leq \delta$;
  \item $2\delta+2\leq C_0,C_1\leq 3\delta+2$. Here $C_0$ is even and $C_1$ is odd.
\end{itemize}
\end{definition}
We remark that our notion of acceptability is a restricted form of acceptability in Cherlin's monograph to exclude
non-primitive cases and cases of infinite diameter.
\begin{theorem}[Cherlin's Admissibility Theorem~\cite{Cherlin2013} (simplified)]
\label{thm:admissible}
Let $(\delta,\allowbreak K_1,\allowbreak K_2,\allowbreak C_0,\allowbreak C_1)$ be an acceptable sequence of parameters (in particular, $\delta\geq 3$). Then
the associated class $\mathcal A^\delta_{K_1,K_2,C_0,C_1}$ is an amalgamation class if and only if
one of the following two groups of conditions is satisfied, where we write $C$ for $\min(C_0,C_1)$
and $C'$ for $\max(C_0,C_1)$:
\begin{enumerate}[label=(\Roman*)]
  \setcounter{enumi}{1}

\setlength\itemsep{0em}
\item\label{II} $C\leq 2\delta+K_1$, and
\begin{itemize}
 \setlength\itemsep{0em}
 \item $C=2K_1+2K_2+1$;
 \item $K_1+K_2\geq \delta$;
 \item $K_1+2K_2\leq 2\delta-1$, and:
\end{itemize}
\begin{enumerate}[label=(II\Alph*)]
\setlength\itemsep{0em}
\item\label{IIa} $C'=C+1$, or
\item\label{IIb} $C'>C+1, K_1=K_2$, and $3K_2=2\delta-1$.
\end{enumerate}
\item\label{III} $C>2\delta+K_1$, and:
\begin{itemize}
 \setlength\itemsep{0em}
 \item $K_1+2K_2\geq 2\delta-1$ and $3K_2\geq 2\delta$;
 \item If $K_1+2K_2=2\delta-1$ then $C\geq 2\delta+K_1+2$;
 \item If $C'>C+1$ then $C\geq 2\delta+K_2$.
\end{itemize}
\end{enumerate}
\end{theorem}
An acceptable sequence of parameters $(\delta,\allowbreak K_1,\allowbreak K_2,\allowbreak C_0,\allowbreak C_1)$ is called {\em admissible} if and only if it satisfies one of the sets of conditions in Theorem~\ref{thm:admissible}.

Let us again remark that this is a simplified version of Cherlin's theorem which only allows for the primitive 3-constrained cases. This is also the reason why the first case has number~\ref{II}; we wanted to keep the case numbering the same as in Cherlin's monograph.

\subsection{Our results}
A \emph{$\delta$-edge-labelled graph} is a graph $\str G=(V,E)$ together with a labelling function $\ell\colon E\rightarrow \{1,2,\ldots,\delta\}$ giving each edge a label. Alternatively, we can treat $\str G$ as a structure in a relational language with symmetric binary relations $R^1, R^2, \ldots, R^\delta$ such that each pair of vertices is in at most one relation. We will denote $\mathcal G^\delta$ the class of all $\delta$-edge-labelled graphs.

If $\str G$ is a complete graph and $\ell$ satisfies the triangle inequality, then $(\str G, \ell)$ is a metric space. If further $\ell$ omits all triangles from Definition~\ref{defn:numerical} for some admissible $(\delta,\allowbreak K_1,\allowbreak K_2,\allowbreak C_0,\allowbreak C_1)$, then we can identify $(\str G,\ell)$ with a member of $\mathcal A^\delta_{K_1,K_2,C_0,C_1}$.

Now, for admissible parameters $(\delta,\allowbreak K_1,\allowbreak K_2,\allowbreak C_0,\allowbreak C_1)$, we define
$$\mathcal G^\delta_{K_1,K_2,C_0,C_1} = \left\{(V,E,\ell)\in \mathcal G^\delta : (\exists (V,d)\in \mathcal A^\delta_{K_1,K_2,C_0,C_1})(d|_E = \ell)\right\},$$
i.e. the class of all $\delta$-edge-labelled graphs one can get from a member of $\mathcal A^\delta_{K_1,K_2,C_0,C_1}$ by deleting some edges. Alternatively $\mathcal G^\delta_{K_1,K_2,C_0,C_1}$ is precisely the class of $\delta$-edge-labelled graphs with the property that one can add labels to the non-edges and get a metric space from $\mathcal A^\delta_{K_1,K_2,C_0,C_1}$.

In~\cite{Aranda2017} it was proved that for every admissible sequence of parameters $(\delta,\allowbreak K_1,\allowbreak K_2,\allowbreak C_0,\allowbreak C_1)$ there is a finite family $\mathcal F$ of $\delta$-edge-labelled cycles such that $\mathcal G^\delta_{K_1,K_2,C_0,C_1} = \Forb(\mathcal F)$ where $\Forb(\mathcal F)$ means the subclass of $\mathcal G^\delta$ such that there is no $\str F\in\mathcal F$ with a homomorphism to some member of the subclass.

In this paper we give an explicit description of $\mathcal F$ for each admissible $(\delta,\allowbreak K_1,\allowbreak K_2,\allowbreak C_0,\allowbreak C_1)$. It is not surprising that there are multiple types of forbidden cycles corresponding to triangles forbidden by different bounds ($K_1, K_2, C_0, C_1$).

\begin{definition}[Forbidden cycles] \label{defn:forbcycles}
Let $(\delta,\allowbreak K_1,\allowbreak K_2,\allowbreak C_0,\allowbreak C_1)$ be an admissible sequence of parameters. By a cycle we mean a $\delta$-edge-labelled graph $(V,E,\ell)$, where the graph $(V,E)$ is a cycle. We say that $(V,E,\ell)$ is a \emph{cycle with distances $d_1,d_2,\ldots,d_k$} if one can order the edges as $E = \{e_1, e_2, \ldots, e_k\}$ such that $\ell(e_i) = d_i$ for every $1\leq i\leq k$. If $\str C = (V,E,\ell)$ is a cycle with distances $d_1, \ldots, d_k$ in this cyclic order, we will write $\str C = (d_1, \ldots, d_k)$. We say that a cycle with distances $d_1,d_2,\ldots,d_k$ \emph{has perimeter $p$} if $p = \sum_i d_i$.

The following will be the building blocks of $\mathcal F$:
\begin{description}
\item[Non-metric cycles] Cycles with edges $a, x_1, x_2, \ldots, x_k$ such that $$a > \sum_{i=1}^k x_i.$$
\item[$C$-cycles] Cycles with distances $d_0, d_1, \ldots, d_{2n}, x_1, \ldots, x_k$ for some $n\geq 0$ such that
$$\sum_{i=0}^{2n} d_i > n(C - 1) + \sum_{i=1}^k x_i.$$
\item[$C_0$-cycles] Cycles of even perimeter with distances $d_0, d_1, d_2, x_1, \ldots, x_k$ for $n\in\{0,1\}$ such that 
$$\sum_{i=0}^{2n} d_i > n(C_0 - 1) + \sum_{i=1}^k x_i.$$
\item[$C_1$-cycles] Cycles of odd perimeter with distances $d_0, d_1, d_2, x_1, \ldots, x_k$ for $n\in\{0,1\}$ such that 
$$\sum_{i=0}^{2n} d_i > n(C_1 - 1) + \sum_{i=1}^k x_i.$$
\item[$K_1$-cycles] \emph{Metric} cycles of odd perimeter with distances $x_1, \ldots, x_k$ such that $$2K_1 > \sum_{i=1}^k x_i.$$
\item[$K_2$-cycles] Cycles of odd perimeter with distances $d_1,\ldots,d_{2n+2}, x_1, \ldots,x_k$ such that $$\sum_{i=1}^{2n+2}d_i > 2K_2 + n(C-1) + \sum_{i=1}^k x_i.$$
\end{description}
Note that the non-metric cycles are precisely the union of $C_0$- and $C_1$-cycles (or the $C$-cycles) for $n=0$. Sometimes we will treat them separately, sometimes the fact that non-metric cycles belong to the $C_x$-cycles family will be useful.
\end{definition}

And now we can state the main result of this paper.
\begin{theorem}\label{thm:main}
Let $(\delta,\allowbreak K_1,\allowbreak K_2,\allowbreak C_0,\allowbreak C_1)$ be an admissible sequence of parameters. Then $\mathcal G^\delta_{K_1,K_2,C_0,C_1} = \Forb(\mathcal F)$, such that $\mathcal F$ is the one of the following:
\begin{enumerate}
\item If $|C_0-C_1|=1$, $\mathcal F$ is the union of all $C$-cycles, $K_1$-cycles, $K_2$-cycles and non-metric cycles;
\item if $|C_0-C_1|>1$ and $\delta > 5$ or the parameters come from Case~\ref{III}, $\mathcal F$ is the union of all $C_0$-cycles, $C_1$-cycles, $K_1$-cycles and $K_2$-cycles, and non-metric cycles; and
\item if $|C_0-C_1|>1$, $\delta = 5$ and the parameters come from Case~\ref{IIb}, $\mathcal F$ is the union of all $C_0$-cycles, $C_1$-cycles, $K_1$-cycles, $K_2$-cycles, the cycle $(5,5,5,5,5)$, and non-metric cycles.
\end{enumerate}
\end{theorem}

\section{The magic completion algorithm}
In the proof we shall rely on some results of~\cite{Aranda2017} which are briefly presented in this section, namely on the magic completion algorithm. The presentation in this paper will be somewhat different (although equivalent) from the presentation in~\cite{Aranda2017} because of different goals --- in~\cite{Aranda2017} the presentation was optimized for proving the correctness of the magic completion algorithm whereas here we want to apply it.

For a $\delta$-edge-labelled graph $\str G = (V, E, \ell)$ we say that a metric space $\str M = (M, d) \in \mathcal A^\delta_{K_1,K_2,C_0,C_1}$ is its \emph{completion} if $V=M$ and $d|_E = \ell$. The magic completion algorithm is an explicit way of looking for a completion of a given $\delta$-edge-labelled graph by setting the length of each missing edge to be as close to some \emph{magic parameter M} as possible.

The following definition is a merge of Definitions~4.3 and~4.4 from~\cite{Aranda2017}.
\begin{definition}[Magic distances]
\label{defn:magicdistance}
Let $M\in\{1,2,\ldots, \delta\}$ be a distance. We say that $M$ is {\em magic} (with respect to $\mathcal A^\delta_{K_1,K_2,C_0,C_1}$) if $$\max\left(K_1, \left\lceil\frac{\delta}{2}\right\rceil\right) \leq M \leq \min\left(K_2,\left\lfloor\frac{C-\delta-1}{2}\right\rfloor\right)$$
and further $M$ satisfies the following two extra conditions:
\begin{enumerate}
\item If the parameters satisfy Case~\ref{III} with $K_1+2K_2 = 2\delta - 1$, then $M>K_1$;
\item if the parameters satisfy Case~\ref{III} and further $C'>C+1$ and $C=2\delta+K_2$, then $M<K_2$.
\end{enumerate}
\end{definition}
\begin{observation}[Lemma~4.2 in~\cite{Aranda2017}]
For every admissible $(\delta,\allowbreak K_1,\allowbreak K_2,\allowbreak C_0,\allowbreak C_1)$ there is a magic distance.
\end{observation}

Magic distances are the safe distances towards which it is possible to optimize in the magic completion algorithm, which we shall present now, but in a different manner than~\cite{Aranda2017}, inspired by the work of the first and third authors on generalised metric spaces~\cite{Hubicka2017sauer,Konecny2018b}.

\begin{definition}[Magic semigroup]
Fix an admissible sequence of parameters $(\delta,\allowbreak K_1,\allowbreak K_2,\allowbreak C_0,\allowbreak C_1)$ and a magic distance $M$. Put $C=\min(C_0, C_1)$. Then define the commutative operation $\oplus\colon [\delta]^2 \rightarrow [\delta]$ as follows:
$$x \oplus y =
  \begin{cases}
    |x-y| & \text{if } |x-y| > M \\
    \min\left(x+y, C-1-x-y\right) & \text{if } \min\left(\ldots\right) < M \\
    M & \text{otherwise}.\end{cases}$$
\end{definition}
It can be proved that $\oplus$ is associative, but we shall not need it.

We say that the triple of vertices $u,v,w$ is a \emph{fork} if the distances between $u$ and $v$ and $v$ and $w$ are defined, while the distance $uw$ is not defined. If $d(u,v)=a$ and $d(v,w)=b$, we also say that $a,b$ is a fork.

If $x\oplus y = x+y$, we say that $x,y$ is completed by the \emph{$d^+$-fork}, if $x\oplus y = |x-y|$, we say that $x,y$ is completed by the \emph{$d^-$-fork} and if $x\oplus y = C-1-x-y$, we say that $x,y$ is completed by the \emph{$d^C$-fork}.

The following fact summarizes several properties of $\oplus$. We shall use these properties implicitly throughout the paper.
\begin{fact}\label{fact:oplus}\leavevmode
\begin{enumerate}
\item $M\oplus x = M$ for all $1\leq x\leq \delta$;
\item if the parameters come from Case~\ref{III}, then $C-1-x-y \geq K_1$;
\item if the parameters come from Case~\ref{IIb}, then $C-1$ is even and thus unless $a\oplus b = M$, the parity of $a\oplus b$ is always the same as the parity of $a+b$;
\item if the parameters come from Case~\ref{III} and $C' > C+1$, then $C-1-x-y\geq K_2-1 \geq M$ (from the extra condition on $M$), hence the $d^C$ fork is never used;
\item if the parameters come from Case~\ref{IIb}, then $C-1-x-y = K_2-1 = M-1$, so the $d^C$ fork is only used for $\delta\oplus\delta = M-1$.
\end{enumerate}
\end{fact}

The magic completion algorithm runs in stages. It orders the distances $\{1,\ldots,\delta\}$ in a particular order as $d_1, \ldots, d_\delta$ and in the $i$-th stage it looks at each fork $x,y,z$ and if $d(x,y)\oplus d(x,z) = d_i$, then it sets $d(y,z)=d_i$. Before stating this formally, we need to present the correct order $d_1, \ldots, d_\delta$.

\begin{definition}[Time function and the magic permutation]
Assume that an admissible sequence of parameters $(\delta,\allowbreak K_1,\allowbreak K_2,\allowbreak C_0,\allowbreak C_1)$ and a magic distance $M$ are fixed. Then define the function $t\colon [\delta] \rightarrow \mathbb N\cup \{\infty\}$ by
$$t(x) =
  \begin{cases}
    2x+1 & \text{if } x<M \\
    2(\delta-x) & \text{if } x > M \\
    \infty & \text{if } x = M.\end{cases}$$
Using this function define the permutation $d_1,\ldots,d_\delta$ of distances $1,\ldots,\delta$ by $t(d_i) \leq t(d_j)$ if and only if $i\leq j$. We will call it the \emph{magic permutation}.
\end{definition}
So, typically (for $\delta$ large enough and $M$ small enough) we will have $d_1 = \delta, d_2 = 1, d_3 = \delta-1, d_4 = 2, \ldots, d_\delta = M$. Now we are ready to state the magic completion algorithm explicitly:
\begin{definition}[The magic completion algorithm]
Assume that an admissible sequence of parameters $(\delta,\allowbreak K_1,\allowbreak K_2,\allowbreak C_0,\allowbreak C_1)$ and a magic distance $M$ are fixed. Let $d_1, d_2, \ldots, d_\delta$ be the corresponding magic permutation.

For a $\delta$-edge-labelled graph $\str G = (V, E, \ell) \in \mathcal G^\delta$ define the sequence $\str G = \str G_0 \subseteq \str G_1 \subseteq \ldots \subseteq \str G_\delta$ of $\delta$-edge-labelled graphs with the vertex set $V$ such that $E(\str G_i) \subseteq E(\str G_{i+1})$, $E(\str G_\delta) = {V\choose 2}$ and $\ell_{\str G_{i+1}}|_{E(\str G_i)} = \ell_{\str G_i}$. We use the following induction rule for $i=0,1,\ldots\delta-1$:

Given $\str G_i = (V, E_i, \ell_i)$, we look at each non-edge $xy$ of $\str G_i$ and each vertex $z\in V$ such that $xz\in E_i$ and $yz\in E_i$. If $\ell_i(x,z)\oplus\ell_i(y,z) = d_{i+1}$, we set $\ell_{i+1}(x,y) = \ell_{i+1}(y,x) = d_{i+1}$. Then, of course, for every $xy\in E_i$ we set $\ell_{i+1}(x,y) = \ell_{i+1}(y,x) = \ell_i(x,y)$. And finally we let $E_{i+1}$ be the set of pairs where $\ell_{i+1}$ is defined.

We say that $\str G_\delta$ is the \emph{magic completion (with parameter $M$) of $\str G$}.
\end{definition}

The following theorem is a crucial result of~\cite{Aranda2017}:
\begin{theorem}[\cite{Aranda2017}]\label{thm:magic}
A $\delta$-edge-labelled $\str G\in \mathcal G^\delta$ is in $\mathcal G^\delta_{K_1,K_2,C_0,C_1}$ if and only if the magic completion of $\str G$ is in $\mathcal A^\delta_{K_1,K_2,C_0,C_1}$.
\end{theorem}
We shall use the contrapositive version of Theorem~\ref{thm:magic}, namely that if the magic completion of $\str G$ is not in $\mathcal A^\delta_{K_1,K_2,C_0,C_1}$, then $\str G \notin \mathcal G^\delta_{K_1,K_2,C_0,C_1}$, i.e. there is no completion of $\str G$ to $\mathcal A^\delta_{K_1,K_2,C_0,C_1}$.

\section{Proof strategy}
We want to prove that every $\str G\in \mathcal G^\delta \setminus \mathcal G^\delta_{K_1,K_2,C_0,C_1}$ contains a homomorphic image of a member of $\mathcal F$ as described in Theorem~\ref{thm:main}. To achieve this, we will take the magic completion of such a $\str G$. By the assumption, the completion is not in $\mathcal A^\delta_{K_1,K_2,C_0,C_1}$, hence contains a forbidden triangle. Then we shall look at the run of the magic completion algorithm and extract a non-completable witness in $\str G$ and observe that it is a homomorphic image of a member of $\mathcal F$.\footnote{This is the same strategy as used in~\cite{Aranda2017} to prove that $\mathcal F$ is finite, here we shall do a finer analysis to get an explicit description.}

Having guessed the family $\mathcal F$, we need to prove three things: That the triangles forbidden in $\mathcal A^\delta_{K_1,K_2,C_0,C_1}$ are precisely the triangles in $\mathcal F$, that $\mathcal F$ is closed on the steps of the magic completion algorithm and that $\mathcal F$ is closed on the inverse steps of the magic completion algorithm.

\begin{definition}[Steps and inverse steps of the magic completion algorithm]
Assume that an admissible sequence of parameters $(\delta,\allowbreak K_1,\allowbreak K_2,\allowbreak C_0,\allowbreak C_1)$ and a magic distance $M$ are fixed.

Let $\str C$ be a cycle with distances $c_1, c_2, \ldots, c_k$ such that edges of lengths $c_i$ and $c_{i+1}$ share a vertex, where we identify $c_{k+1}=c_1$ (that is, $\str C = (c_1,\ldots,c_k)$). Let $i$ be the smallest $i$ such that there is some $1\leq j\leq k$ with $c_j\oplus c_{j+1} = d_i$, that is, the first stage of the magic completion algorithm where something would happen with $\str C$.

Take an arbitrary $1\leq j \leq k$ with $c_j\oplus c_{j+1} = d_i$ and let $\str C'$ be the cycle with edges of lengths $c_1, \ldots, c_{j-1}, d_i, c_{j+2}, \ldots, c_k$ in this cyclic order. Then we say that one can \emph{get $\str C'$ from $\str C$ by a step of the magic completion algorithm} and that one can \emph{get $\str C$ from $\str C'$ by an inverse step of the magic completion algorithm}.
\end{definition}

\begin{lemma}\label{lem:aboutf}
Assume that an admissible sequence of parameters $(\delta,\allowbreak K_1,\allowbreak K_2,\allowbreak C_0,\allowbreak C_1)$ and a magic distance $M$ are fixed. Let $\mathcal F$ be a family of $\delta$-edge-labelled cycles such that the following holds:
\begin{enumerate}
\item The three-vertex members of $\mathcal F$ are precisely the triangles forbidden in $\mathcal A^\delta_{K_1,K_2,C_0,C_1}$;
\item for every $\str C\in \mathcal F$ and every $\str C'$ such that one can get $\str C'$ from $\str C$ by a step of the magic completion algorithm it holds that $\str C'\in \mathcal F$ (i.e. $\mathcal F$ is \emph{closed on the steps of the magic completion algorithm}); and
\item for every $\str C\in \mathcal F$ and every $\str C'$ such that one can get $\str C'$ from $\str C$ by an inverse step of the magic completion algorithm it holds that $\str C'\in \mathcal F$ (i.e. $\mathcal F$ is \emph{closed on the inverse steps of the magic completion algorithm}).
\end{enumerate}
Then $\mathcal G^\delta_{K_1,K_2,C_0,C_1} = \Forb(\mathcal F)$. In other words, $\mathcal G^\delta_{K_1,K_2,C_0,C_1}$ is the subclass of $\mathcal G^\delta$ containing those graphs such that there is no homomorphism from a member of $\mathcal F$ to them.
\end{lemma}
\begin{proof}
By Lemma~4.18 in~\cite{Aranda2017} there is a finite family of $\delta$-edge-labelled cycles $\mathcal O$ such that $$\mathcal G^\delta_{K_1,K_2,C_0,C_1} = \Forb(\mathcal O).$$ We shall prove that no member of $\mathcal F$ has a completion in $\mathcal A^\delta_{K_1,K_2,C_0,C_1}$ and that whenever a cycle $\str C$ has no completion in $\mathcal A^\delta_{K_1,K_2,C_0,C_1}$, then $\str C\in \mathcal F$. This implies $\mathcal F=\mathcal O$.

Take an arbitrary $\str C\in \mathcal F$ and take its magic completion. If one looks at the run of the magic completion algorithm, it consists of many steps of the magic completion algorithm run in parallel. By focusing on just one \textit{thread} (i.e. in each stage we only add such edges that we end up with a smaller cycle with some triangles cut out), it is easy to see that it is just a sequence of steps of the magic completion algorithm. As $\mathcal F$ is closed on them, we eventually arrive to a triangle from $\mathcal F$, but it is, by the assumption, forbidden in $\mathcal A^\delta_{K_1,K_2,C_0,C_1}$, hence $\str C$ has no completion in $\mathcal A^\delta_{K_1,K_2,C_0,C_1}$ (using Theorem~\ref{thm:magic}).

Now take an arbitrary $\str G\in \mathcal G^\delta \setminus \mathcal G^\delta_{K_1,K_2,C_0,C_1}$ and let $\str G'$ be its magic completion. As $\str G\in \mathcal G^\delta \setminus \mathcal G^\delta_{K_1,K_2,C_0,C_1}$, there is a forbidden triangle in $\str G'$ and by our assumption this triangle is in $\mathcal F$. In the same way as in the proof of Lemma~4.18 in~\cite{Aranda2017}, we can backtrack the run of the magic completion algorithm monitoring which \textit{forks} caused the forbidden triangle to appear in $\str G'$ and eventually we arrive to a homomorphic image of a cycle in $\str G$ with no completion in $\mathcal A^\delta_{K_1,K_2,C_0,C_1}$. As $\mathcal F$ is closed on the inverse steps of the magic completion algorithm, this witness is going to be in $\mathcal F$, which is what we wanted to prove.
\end{proof}

By simply checking the definition of $\mathcal F$, one can see that the first point of Lemma~\ref{lem:aboutf} holds: $K_1$-forbidden triangles are the 3-vertex $K_1$-cycles, $K_2$-forbidden triangles are the 3-vertex $K_2$-cycles, non-metric triangles are the 3-vertex $C_0$- and $C_1$-cycles (or $C$-cycles) with $n=0$ and $C_0$- resp. $C_1$-triangles are the 3-vertex  $C_0$- and $C_1$-cycles respectively (or, together, $C$-cycles) with $n=1$.

To prove closedness of $\mathcal F$ on steps and inverse steps of the magic algorithm, we need to do some case-work and separately for each type of forbidden cycle (and often even separately for different cases of admissible parameters) check that indeed both the steps and inverse steps of the magic completion algorithm produce a cycle from $\mathcal F$ when run on the given type of forbidden cycle.

When analyzing the inverse steps, edges of length $M$ are quite problematic, because a lot of different pairs of distances $\oplus$-sum up to $M$ (including, say, $M\oplus M = M$). In~\cite{Aranda2017} this was dealt with by the following observation
\begin{lemma}[Lemma 4.4 in~\cite{Aranda2017}]\label{lem:tension3}
Let $\str{G}\in \mathcal G^\delta$ and $\overbar{\str{G}}$ be its completion with magic parameter $M$. If there is a forbidden triangle (w.r.t. $\mathcal A^\delta_{K_1,K_2,C_0,C_1}$) or a triangle with perimeter at least $C$ in $\overbar{\str{G}}$ with an edge of length $M$, then this edge is also present in $\str{G}$.
\end{lemma}
Thanks to this lemma combined with the observation that $M\oplus a = M$ for every $a$ one knows that the edges of length $M$ never participate in any inverse steps.

In this paper we need to find a generalisation of this lemma:
\begin{definition}[Tension]
Let $\str{C} \in \mathcal G^\delta$ be a cycle. We say that there is a \emph{$\oplus$-tension} (often just called \emph{tension}) in $\str C$ if there are two neighbouring edges of $\str{C}$ with lengths $a$ and $b$ such that $a\oplus b \neq M$.
\end{definition}
\begin{lemma}
Let $\str{C} \in \mathcal G^\delta$ be a cycle with a tension. Then the following hold:
\begin{enumerate}
\item\label{lem:tension:1} Let $\str{C}'\in\mathcal G^\delta$ be a cycle which one can get from $\str{C}$ by a step of the magic completion algorithm and let $e$ be the newly added edge. Then $e\neq M$.

\item\label{lem:tension:2} Let $\str{C}'\in\mathcal G^\delta$ be a cycle which one can get from $\str{C}$ by an inverse step of the magic completion algorithm and let $e$ be the edge of $\str{C}$ which was replaced by a fork in $\str{C}'$. Then $e\neq M$.
\end{enumerate}
\end{lemma}
\begin{proof}
We prove both points by contradiction. Suppose that $\str{C}$ and $\str{C}'$ give such a contradiction.
\begin{enumerate}
\item Since $\str{C}$ has a tension, there are vertices $u,v,w$ such that $uv$ and $vw$ are edges of $\str{C}$ and $uw$ is not an edge of $\str{C}$ (otherwise $\str{C}$ is a triangle and there is no $\str{C}'$). But then $t(d(u,v) + d(v,w)) < t(M)$ which is a contradiction with the definition of a step of the magic completion algorithm.
\item One can just repeat the previous paragraph with the roles of $\str{C}$ and $\str{C}'$ switched after noticing that if $\str C$ had a tension, then $\str C'$ has a tension (as $M\oplus a=M$ for every $a$).
\end{enumerate}
\end{proof}

For each type of forbidden cycles we prove that it has tension given that it has at least four vertices and then we can simply ignore edges of length $M$ for the inverse steps and forks which $\oplus$-sum to $M$ for the direct steps (for triangles this follows from Lemma~\ref{lem:tension3}).

\section{$K_1$-cycles}
In the whole section we let $\str{C}$ be a $K_1$-cycle with edges $x_1, \ldots, x_k$.

\begin{lemma}\label{lem:structure:k1}
For every $1\leq i\leq k$ it holds that $x_i < K_1$ and hence $x_i < M$.
\end{lemma}
\begin{proof}
Take an arbitrary $x_i$. Then, as $\str{C}$ is metric, we have $x_i\leq \sum_{j\neq i} x_j$ and hence $2x_i\leq \sum_i x_i < 2K_1$.
\end{proof}

\begin{lemma}\label{lem:tension:k1}
If $\str C$ has at least 4 vertices, then there is a pair of neighbouring edges $a,b$ such that $a+b < K_1$, hence $\str C$ has a tension. 
\end{lemma}
\begin{proof}
For convenience identify $x_{k+1} = x_1$. Look at $\sum_{i=1}^k (x_i+x_{i+1})$. Clearly
$$\sum_{i=1}^k (x_i+x_{i+1}) = 2\sum_i x_i < 4K_1.$$
Let $i$ be such that $x_i+x_{i+1}$ is the smallest possible. Then $k(x_i+x_{i+1}) \leq \sum_{i=1}^k (x_i+x_{i+1}) < 4K_1$, hence $x_i\oplus x_{i+1} = x_i+x_{i+1} < K_1 \leq M$.
\end{proof}

\begin{lemma}
If $\str C'$ can be obtained from $\str C$ by a step of the magic completion algorithm, then $\str C'$ is a $K_1$-cycle.
\end{lemma}
\begin{proof}
Let $u,v,w$ be the vertices of $\str C$ such that $uv$ is not an edge of $\str C$ an it is an edge of $\str C'$. Denote $a=d_\str{C}(u,v)$ and $b=d_\str{C}(v,w)$. By Lemma~\ref{lem:tension:k1} we know that $a\oplus b\neq M$ and clearly $a\oplus b \neq |a-b|$.

Because $a,b < K_1$ by Lemma~\ref{lem:structure:k1}, we have $C-1-a-b > C-1-2K_1$. For Case~\ref{III} this gives $C-1-a-b > 2\delta-K_1 \geq \delta \geq M$, for Case~\ref{II} this gives $C-1-a-b > 2K_2 > M$. Hence $a\oplus b\neq C-1-a-b$, thus $a\oplus b = a+b$. And then clearly $\str C'$ is a $K_1$-cycle.
\end{proof}

We know that if $\str C$ has at least four vertices, then it has a tension. The tension clearly does not involve any edge of length $M$. So this means that when $\str C'$ can be obtained from $\str C$, then they do not differ by \textit{expanding} an edge of length $M$. If $\str C$ has three vertices, then the same conclusion is given precisely by Lemma~\ref{lem:tension3}.

\begin{lemma}
If $\str C'$ can be obtained from $\str C$ by an inverse step of the magic completion algorithm which expanded edge $p$ to edges $q,r$, then $\str C'$ is a $K_1$- or $K_2$-cycle.
\end{lemma}
\begin{proof}
In Case~\ref{III} we have $C-1-x-y \geq K_1>p$, hence $p=q\oplus r = q+r$. And this clearly preserves parity and the $K_1$ inequality.

In Case~\ref{II} if $q\oplus r = q+r$, then as in Case~\ref{III} $\str C'$ is again a $K_1$-cycle. Otherwise $p = q\oplus r = C-1-q-r = 2K_2 + 2K_1 - q-R$. We know that $p + \sum_{x_i\neq p} x_i$ is odd and smaller than $2K_1$. Thus
$$2K_2 + 2K_1 - q-r + \sum_{x_i\neq p} x_i < 2K_1$$
and this sum is odd. But one can rearrange the terms as
$$q+r > 2K_2 + \sum_{x_i\neq p} x_i,$$
which means that $\str C'$ is a $K_2$-cycle.
\end{proof}

\section{Metric $K_2$-cycles with parameters from Case~\ref{III}}
Observe that with parameters from Case~\ref{III} we necessarily have $n=0$: The $K_2$ inequality states that
$$\sum_{i=1}^{2n+2}d_i > 2K_2 + n(C-1) + \sum x_i,$$
which can be rearranged as
$$d_1+d_2 > 2K_2 + \sum x_i + \sum_{i=1}^n\left(C-1-d_{2i+1}-d_{2i+2}\right).$$
And clearly $C-1-d_{2i+1}-d_{2i+2} \geq C-1-2\delta$. So we have
$$d_1+d_2 > 2K_2 + n(C-1-2\delta).$$
In Case~\ref{III} it holds that $C > 2\delta+K_1$ and $2K_2+K_1\geq 2\delta-1$, and when $2K_2+K_1 = 2\delta-1$, then $C\geq 2\delta+K_1+2$. This implies that $n=0$.

Therefore in the whole section we can let $\str{C}$ be a $K_2$-cycle with edges $a, b, x_1, \ldots, x_k$ such that $\str C$ is metric. We also assume that the parameters belong to Case~\ref{III}.

\begin{lemma}\label{lem:structure:k2III}
It holds that $a,b > K_2$ and $\sum x_i \leq K_1$ and also $\sum x_i < K_2$. Furthermore if $\sum x_i = K_1$, then $C\geq 2\delta + K_1 + 2$ and $M > K_1$, so it always holds $\sum x_i < M$.
\end{lemma}
\begin{proof}
Surely $2\delta \geq a+b > 2K_2 + \sum x_i$, but from the admissibility conditions we get $3K_2 \geq 2\delta$ and $2K_2 + K_1 \geq 2\delta-1$. The bounds on $\sum x_i$ now follow.

Without loss of generality suppose $a \geq b$ and for a contradiction suppose $b \leq K_2$. Then
$$a+b > 2K_2 + \sum x_i\geq 2b+\sum x_i,$$
hence $a > b+\sum x_i$. But this means that $\str C$ is non-metric, which is a contradiction with the assumptions.
\end{proof}

\begin{lemma}\label{lem:tension:k2III}
If $\str C$ has at least 4 vertices, then $\str C$ has a tension.
\end{lemma}
\begin{proof}
If there are some $x_i, x_j$ which are adjacent, then $x_i\oplus x_j = x_i+x_j < M$ and we found a tension. Otherwise $\str C=(a,x_1,b,x_2)$. Suppose without loss of generality that $a\geq b$ and $x_1\leq x_2$. This means that $a > K_2 + x_1$, hence $a\oplus x_1 = a-x_1 > K_2 \geq M$, again a tension.
\end{proof}

\begin{lemma}
If $\str C'$ can be obtained from $\str C$ by a step of the magic completion algorithm, then $\str C'$ is a $K_2$-cycle.
\end{lemma}
\begin{proof}
Clearly $a\oplus x_i$ and $b\oplus x_i$ use the $d^-$-fork or are equal to $M$ for every $i$. and $d^-$ fork on these edges preserves the $K_2$ inequality. Also $x_i\oplus x_j = x_i+x_j$ for every $i\neq j$, which again preserves the $K_2$ inequality.

It remains to check what would happen if $\str C'$ differed from $\str C$ by replacing $a,b$ by $a\oplus b$. This can only happen if $a\oplus b = C-1-a-b$. But this would mean that $a,b$ are adjacent. And thus there are some $x_i, x_j$ adjacent. If $x_i+x_j<K_1$, then it is a contradiction with time, as $C-1-a-b\geq K_1$ and thus $t(x_i+x_j) < t(C-1-a-b)$. So the only possibility is $x_i+x_j = K_1$. But then $C\geq 2\delta+K_1+2$, hence $C-1-a-b\geq K_1+1$, again a contradiction with time.
\end{proof}

\begin{lemma}
If $\str C'$ can be obtained from $\str C$ by an inverse step of the magic completion algorithm which expanded edge $p$ to edges $q, r$. Then $\str C'$ is a $K_2$-cycle.
\end{lemma}
\begin{proof}
We know that $a,b>K_2$ and $x_i\leq K_1$ for every $i$ with equality implying $C\geq 2\delta+K_1+2$. This means that if $p$ is one of $a,b$, say, $a$, then $q\oplus r = |q-r|$, say $q\geq r$, hence $q\oplus r = q-r$. But then $q+b > 2K_2 + r + \sum x_i$ and $\str C'$ is again a $K_2$-cycle.

Otherwise $p$ is $x_i$ for some $i$. Then $q\oplus r = q+r = x_i$. And thus $a+b > 2K_2 + q+r+\sum_{x_j\neq x_i} x_j$ and $\str C'$ is a $K_2$-cycle.
\end{proof}

\section{Non-$C$ $K_2$-cycles with parameters from Case~\ref{II}}
In the whole section we let $\str{C}$ be a $K_2$-cycle with edges $d_1, \ldots, d_{2n+2}, x_1, \ldots, x_k$ such that $\str C$ is not a $C$-cycle. We also assume that the parameters belong to Case~\ref{II}.

\begin{lemma}\label{lem:stupid:k2II}
If $\str D$ is a $C$-cycle and the parameters come from Case~\ref{IIb}, then either $n\leq 1$, or $\delta=5$ and $\str D = (5,5,5,5,5)$.
\end{lemma}
\begin{proof}
In Case~\ref{IIb} we have $C=2\delta+K_2 = 4K_2+1$. Suppose that $n\geq 2$ and that the edges of $\str D$ are named as in the $C$-inequality. Then
$$(2n+1)\delta\geq \sum d_i > 2n\delta + n(K_2 - 1) + \sum x_i \geq 2n\delta + n(K_2 - 1),$$
or
$$\delta > n(K_2-1).$$
However, in Case~\ref{IIb} we have $K_2 = \frac{2\delta-1}{3}$, hence $3\delta > n(2\delta-4)$. If $n\geq 3$, we get $\delta<4$, which is absurd in Case~\ref{IIb}. If $n=2$, we get $\delta < 8$, hence $\delta=5$. For $\delta=5$ all the estimates are actually equalities and it follows that $\str D=(5,5,5,5,5)$.
\end{proof}

Lemma~\ref{lem:stupid:k2II} implies that if the parameters come from Case~\ref{IIb} and $\str D$ is a $C$-cycle of odd perimeter, then $\str D$ is a $C_1$-cycle or the special case $(5,5,5,5,5)$ (as $C=C_1$ in Case~\ref{IIb}), in both cases it is forbidden also for different reasons that the $K_2$-inequality and hence it makes sense to assume here that $\str C$ is not a $C$-cycle (we will deal with those later).

\begin{lemma}\label{lem:structure:k2II}
It holds that $d_i > K_2$ for all $i$, $x_i < K_1$ for all $i$ and $\sum x_i < 2K_1$.
\end{lemma}
\begin{proof}
For a contradiction suppose that (without loss of generality) $d_{2n+2} \leq K_2$. But then
$$d_1 + \ldots + d_{2n+1} > n(C-1) + \sum x_i + 2K_2 - d_{2n+2} \geq n(C-1) + \sum x_i + d_{2n+2},$$
which means that $\str C$ is a $C$-cycle, which is a contradiction.

Now suppose that, say, $x_k \geq K_1$. Then 
$$\sum d_i + x_k > n(C-1) + \sum_{i=1}^{k-1} x_i + 2K_2 + 2x_k \geq n(C-1) + \sum_{i=1}^{k-1} x_i + C-1,$$
hence $\str C$ is again a $C$-cycle, which is a contradiction.

Finally we have
$$2(n+1)\delta \geq \sum d_i > n(2K_2+2K_1) + 2K_2 + \sum x_i.$$
As $K_2 + K_1 \geq \delta$, it follows that $\sum x_i < 2K_1$.
\end{proof}

\begin{lemma}\label{lem:tension:k2II}
If $\str C$ has at least 4 vertices, then $\str C$ has a tension.
\end{lemma}
\begin{proof}
Without loss of generality we can assume that $d_1$ is the smallest among $d_i$'s. Then for every $2\leq j\leq 2n+2$ it holds that
$$2n\delta + d_1 + d_j \geq \sum d_i > 2K_2 + n(C-1) + \sum x_i\geq 2K_2 + 2n\delta + \sum x_i,$$
hence
$$d_1+d_j > 2K_2 + \sum_{i=1}^k x_i.$$
We can also assume that $x_1$ is the largest among $x_i$'s. From this we get that $d_j > K_2+\frac{\sum x_i}{2} \geq K_2 + x_i$ for every $2\leq i\leq k$ (if $k\geq 2$). So if $x_i$ and $d_j$ are adjacent for $i,j\geq 2$, we have $d_j\oplus x_i = d_j - x_i > K_2 \geq M$, a tension. Suppose this does not happen.

Then the only adjacent $d_j$ and $x_i$ can be $d_1$ and $x_i$ for some $i$ and $x_1$ and $d_j$ for some $j$. As there are at least $d_1$ and $d_2$, this implies that all $x_i$'s for a contiguous segment in $\str C$. Either this segment has length zero or one, or it has $x_1$ on one end and it neighbours with $d_1$ on the other.

So we can enumerate the $d_i$'s and $x_i$'s such that $\str C = (x_1, \ldots, x_k, d_1, d_2, \ldots, d_{2n+2})$. Notice that this makes sense even if $k\in \{0,1\}$.

If there are some $d_i, d_j$ which are adjacent and $C-1-d_i-d_j < M$, then we found a tension. So suppose this does not happen. This means that we have $d_i+d_{i+1}\leq C-1-M$ for $1\leq i\leq 2n+1$.

If there are no $x_i$'s, then also $d_1$ and $d_{2n+2}$ are adjacent. Hence we get also $d_1+d_{2n+2}\leq C-1-M$. If we sum all these inequalities, we get
$$2\sum d_i \leq (2n+2)(C-1) - (2n+2)M.$$
On the other hand we know that
$$2\sum d_i > 4K_2 + 2n(C-1).$$
Combining these two inequalities gives $(2n+2)(C-1) > 4K_2 + 2n(C-1)+(2n+2)M$, or $2(C-1) > 4K_2+(2n+2)M\geq 4K_2 + 4K_1$, as $M\geq K_1$ and clearly $n\geq 1$ as $\str C$ has at least four vertices. But $C-1 = 2K_2 + 2K_1$, which gives a contradiction.

Otherwise $k\geq 1$ and we have some $x_i$'s. Then we know that $d_1-x_k\leq M$ and $d_{2n+2}-x_1\leq M$ (otherwise we would have a tension). Summing up these inequalities together with $d_i+d_{i+1}\leq C-1-M$ for $1\leq i\leq 2n+1$ we get
$$2\sum d_i \leq (2n+1)(C-1) - (2n-1)M + x_1 + x_k.$$
And we know that
$$2\sum d_i > 4K_2 + 2n(C-1) + 2\sum x_i.$$

Combining these inequalities we get
$$(2n+1)(C-1) > 4K_2 + 2n(C-1) + (2n-1)M + 2\sum x_i - x_1 - x_k.$$
or
$$C-1 > 4K_2 + (2n-1)M + 2\sum x_i - x_1 - x_k.$$

If $n\geq 1$, then as $2\sum x_i - x_1 - x_k\geq 0$ (note that it is true even for $k=1$), we get $C-1 > 4K_2$, but it is absurd as $C-1=2K_2+2K_1\leq 4K_2$.
If $n = 0$, then as $\str C$ has at least four vertices, we have $k\geq 2$. It means that $2\sum x_i - x_1 - x_k\geq x_1+x_2$ and thus the inequality simplifies to
$$2K_2 + 2K_1 > 4K_2 - M + x_1+x_2.$$
If $x_1+x_2<M$, then they give a tension, hence $x_1+x_2\geq M$ and the inequality is, again, contradictory.
\end{proof}

\begin{lemma}
If $\str C'$ can be obtained from $\str C$ by a step of the magic completion algorithm, then $\str C'$ is a $K_1$- or $K_2$-cycle.
\end{lemma}
\begin{proof}
Clearly $d_i\oplus x_j \in \{d_i-x_j, M\}$ for every $i$ and $j$, $M$ is not used due to tension and $d_i-x_j$ preserves the $K_2$ inequality. Also $x_i\oplus x_j\in \{M,x_i+x_j\}$ for every $i\neq j$, and $x_i+x_j$ again preserves the $K_2$ inequality.

It remains to check what would happen if $\str C'$ differed from $\str C$ by replacing $d_i, d_j$ by $d_i\oplus d_j \in \{M, C-1-d_i, d_j\}$. From tension we know that actually $d_i\oplus d_j = C-1-d_i-d_j$.

If $n\geq 1$, then $\str C'$ is again a $K_2$-cycle with $n'=n-1$.

Otherwise $n=0$, then $d_i=d_1$ and $d_j=d_2$ and we know that $d_1\oplus d_2 = C-1-d_1-d_2= 2K_2 + 2K_1 - d_1-d_2 < 2K_1 - \sum x_i$ as $d_1+d_2>2K_2+\sum x_i$. But then $(d_1\oplus d_2) + \sum x_i < 2K_1$. And as $d_1+d_2$ has the same parity as $2K_2+2K_1-d_1-d_2$, $\str C'$ is a $K_1$-cycle.
\end{proof}

\begin{lemma}
If $\str C'$ can be obtained from $\str C$ by an inverse step of the magic completion algorithm which expanded edge $p$ to edges $q, r$, then $\str C'$ is a $K_2$-cycle.
\end{lemma}
\begin{proof}
We know that $d_i>K_2$ and $x_j < K_1$ for every $i$ and $j$. This means that if $p=d_i$ for some $i$, then $q\oplus r = |q-r|$, say $q\geq r$, hence $d_i=q\oplus r = q-r$. But then $q+\sum_{j\neq i}d_j > n(C-1) + 2K_2 + r + \sum x_i$ and $\str C'$ is again a $K_2$-cycle.

Otherwise $p$ is $x_i$ for some $i$. Then $q\oplus r \in \{q+r, C-1-q-r\}$. If $q\oplus r = q+r$, then $\sum d_i > n(C-1) + 2K_2 + q+r+\sum_{x_j\neq x_i} x_j$ and $\str C'$ is a $K_2$-cycle. If $x_i=C-1-q-r$, then
$$\sum d_i + q + r > (n+1)(C-1) + \sum_{j\neq i} x_j$$
and $\str C'$ is a $K_2$-cycle with $n'=n+1$.
\end{proof}

\section{Non-metric cycles}
In the whole section we let $\str{C}$ be a non-metric cycle with edges $a, x_1, \ldots, x_k$ such that $a>\sum x_i$.
\begin{lemma}\label{lem:tension:nonm}
If $\str C$ has at least 4 vertices, then it has a tension.
\end{lemma}
\begin{proof}
If $a-x_1 > M$, then we found a tension. Otherwise $x_2+\ldots+x_k < a-x_1 \leq M$, so $x_i+x_{i+1} < M$ for every $2\leq i < k$.
\end{proof}

\begin{lemma}
If $\str C'$ can be obtained from $\str C$ by a step of the magic completion algorithm, then $\str C'$ is a non-metric cycle.
\end{lemma}
\begin{proof}
As $\str C$ has at least four vertices, there is tension.

First suppose that the completed fork was $x_i, x_{i+1}$. If $x_i\oplus x_{i+1} \in \{x_i+x_{i+1}, |x_i-x_{i+1}|\}$ then $\str C'$ is still non-metric (the second possibility actually never happens). We know that $x_i+x_j < a \leq \delta$, so $C-1-x_i-x_j >C-1-\delta \geq \delta$, so $x_i\oplus x_{i+1}\neq C-1-x_i-x_j$.

Otherwise the completed fork was without loss of generality $a,x_1$. If $a\oplus x_1 \in \{|a-x_1|, a+x_1\}$ then $\str C'$ is still non-metric (the second possibility again actually never happens). We know that $x_1+x_2+x_3 < a \leq \delta$, so $a+x_1+x_2+x_3 < 2\delta$. But this means that $x_2+x_3 < 2\delta-a-x_1 < C-1-a-x_1$, hence $t(x_2+x_3) < t(C-1-a-x_1)$, so $a\oplus x_1 \neq C-1-a-x_1$.
\end{proof}

\begin{lemma}
If $\str C'$ can be obtained from $\str C$ by an inverse step of the magic completion algorithm which expanded edge $p$ to edges $q, r$, then $\str C'$ is a non-metric cycle or a $C$-cycle with $n=1$.

Moreover, if $C'>C+1$, then in Case~\ref{III} $\str C'$ is always a non-metric cycle and in Case~\ref{IIb} it holds that $\str C'$ is a $C_1$-cycle.
\end{lemma}
\begin{proof}
If $p=x_i$ for some $i$ and $p = q\oplus r = q+r$, then $\str C'$ is non-metric. The same holds is $p=a$ and $q\oplus r = q-r$.

It never happens that $p=a$ and $q\oplus r \in \{q+r, C-1-q-r\}$, because in both cases we have $x_1+x_2 < a = q\oplus r$, a contradiction with time.

It also never happens that $p=x_i$ for some $i$ and $q\oplus r = q-r$, because there is $j\neq i$ such that $x_j$ is adjacent to $a$. And then obviously $a-x_j > x_i > M$, so $t(a-x_j)<t(x_i)$, which is a contradiction.

The last possibility is $p=x_i$ for some $i$ and $p=q\oplus r = C-1-q-r$. If $C'=C+1$, then we have
$$a+q+r > C-1 + \sum_{j\neq i}x_j,$$
i.e. $\str C'$ is a $C$-cycle.

If the parameters come from Case~\ref{III} and $C' > C+1$, then we have already observed that $C-1-x-y\geq M$, hence this never happens.

So it remains to verify what happens when $p=x_i = q\oplus r = C-1-q-r$ when the parameters come from Case~\ref{IIb}. And this is unfortunately going to need some more case analysis.

In that case there is only one possibility, namely $x_i = K_1-1 = K_2-1=M-1$ and $q=r=\delta$. As in Case~\ref{II} the $d^C$-fork preserves parity ($C=C_1$), then if the perimeter of $\str C$ was an odd number, then also the perimeter of $\str C'$ is odd and
$$a+q+r > C-1 + \sum_{j\neq i}x_j,$$
hence $\str C'$ is a $C_1$-cycle.

If the perimeter of $\str C$ was an even number, then also $a-\sum x_i$ is an even number, so in particular $a-\sum x_i\geq 2$. As $x_i = M-1$, we know that $a-x_j\geq M+1$ for all $j\neq i$. But $t(M-1) = 2M-1$ while $t(M+1) = 2\delta - 2M - 2 \leq 2M - 2$ as $M\geq \frac{\delta}{2}$. Hence $t(a-x_j) < t(x_i)$, so this is a contradiction with time.
\end{proof}

\section{$C_0$- and $C_1$- cycles with $n = 1$, Case~\ref{III} and $C'>C+1$}
In the whole section we let $\str{C}$ be a $C_0$- or $C_1$- cycle with edges $d_0, d_1, d_2, x_1, \ldots, x_k$. Suppose further that the parameters come from Case~\ref{III} with $C'>C+1$. As we know that in this case the $d^C$-fork is never used, it follows that $x\oplus y$ preserves parity unless $x\oplus y=M$.

In the remainder of section we can assume that $\str C$ is, say, a $C_0$-cycle, for $C_1$-cycles the same proofs will work.

\begin{lemma}\label{lem:structure:CIIIb}
It holds that $d_i\geq K_2$ for every $i$ and $\sum x_i \leq K_1$.
\end{lemma}
\begin{proof}
As $d_0+d_1+d_2 \geq C \geq 2\delta + K_2$, the first part follows.

Suppose that $\sum x_i > K_1$. Then
$$3\delta \geq \sum d_i \geq C + \sum x_i > 2\delta+K_2+K_1.$$
But $2K_2+K_1\geq 2\delta-1$, so we get
$$3\delta> 2\delta+(2\delta-1-K_2),$$
which implies $K_2\geq\delta$.

But if $K_2 = \delta$, then $C\geq 3\delta$, which means that $k=0$, i.e.~there are no $x_i$'s, so trivially $0 = \sum x_i \leq K_1$.
\end{proof}

\begin{lemma}\label{lem:tension:CIIIb}
If $\str C$ has at least 4 vertices, then it has a tension.
\end{lemma}
\begin{proof}
Without loss of generality suppose that $d_0$ is adjacent to $x_i$ and $d_1$ is adjacent to $x_j$ (it is possible that $i=j$). If $d_0-x_i>M$ or $d_1-x_j>M$, we have found a tension. If this does not happen, then $d_0\leq M+x_i$ and $d_1\leq M+x_j$. Thus we get
$$d_2 + 2M + x_i + x_j \geq d_0+d_1+d_2 \geq C+\sum x_i\geq C+x_j,$$
that is
$$d_2 \geq C-2M-x_i.$$
We know that $x_i\leq \sum x_i \leq \sum d_i - C\leq 3\delta-C$, so
$$d_2\geq 2C-2M-3\delta\geq 4\delta + 2K_2-2M - 3\delta.$$
But from the conditions on $M$ we get $M<K_2$, which means $d_2\geq \delta + 2(K_2-M) > \delta$, a contradiction.

Note that the previous argument holds even for $k=1$ (then the estimate $\sum x_i \geq x_j$ is tight).
\end{proof}

\begin{lemma}
If $\str C'$ can be obtained from $\str C$ by a step of the magic completion algorithm, then $\str C'$ is $C_0$-cycle.
\end{lemma}
\begin{proof}
There are three possibilities. Clearly $x_i+x_j\leq \sum x_i\leq K_1\leq M$, so $x_i\oplus x_j = x_i+x_j$, which preserves the $C_0$ inequality.

As $d_i\geq K_2$ for every $i$, it follows that $d_i\oplus d_j = M$, so this is never used in the step.

Finally $d_i\oplus x_i\in \{M, d_i-x_i\}$, we know that there is a tension in $\str C$, so $d_i\oplus x_i = d_i-x_i$, which again preserves the inequality.
\end{proof}

\begin{lemma}
If $\str C'$ can be obtained from $\str C$ by an inverse step of the magic completion algorithm which expanded edge $p$ to edges $q, r$, then $\str C'$ is a $C_0$-cycle.
\end{lemma}
\begin{proof}
If $p=d_i$ for some $i$, then necessarily $q\oplus r = q-r$ and $\str C'$ is indeed a $C_0$-cycle. Otherwise $p=x_i$ for some $i$ and $q\oplus r = q+r$, so $\str C'$ is again a $C_0$-cycle.
\end{proof}

\section{$C_0$-cycles and $C_1$-cycles with $n = 1$ or the cycle $(5,5,5,5,5)$, Case~\ref{IIb}}\label{sec:c_1iib}
In the whole section we suppose that the parameters come from Case~\ref{IIb} and let $\str{C}$ be a $C_0$- or $C_1$-cycle with edges $d_0, d_1, d_2, x_1, \ldots, x_k$ or the cycle $(5,5,5,5,5)$ if $\delta=5$. It holds that $a\oplus b = C-1-a-b$ only if $a=b=\delta$, that is, $\delta\oplus\delta = K_1-1=K_2-1=M-1$. As $C=2K_1+2K_2+1$, we know that the $d^C$ fork preserves parity, so $a\oplus b$ preserves parity unless $a\oplus b = M$.

\begin{lemma}\label{lem:structure:CIIb}
Suppose than $\str C\neq (5,5,5,5,5)$. Then $d_i\geq K_2$ for every $i$ and $\sum x_i \leq \delta-K_2 < K_1$.
\end{lemma}
\begin{proof}
We have $C = 3K_2 + K_2 + 1 = 2\delta-1 + K_2 + 1 = 2\delta + K_2$. And
$$\sum_{i=0}^2 d_i \geq C+\sum x_i \geq 2\delta+K_2+\sum x_i.$$
From this the statement follows.
\end{proof}

\begin{lemma}\label{lem:tension:CIIb}
If $\str C$ has at least 4 vertices, then it has a tension.
\end{lemma}
\begin{proof}
If $\delta=5$ and $\str C = (5,5,5,5,5)$, then $5\oplus 5 = M-1$, a tension.

Otherwise without loss of generality suppose that $d_0$ is adjacent to $x_i$ and $d_1$ is adjacent to $x_j$ (it can happen that $i=j$). If $d_0-x_i > M$ or $d_1-x_j > M$, we found a tension. Otherwise we get
$$2M+x_i+x_j+d_2 \geq \sum d_i \geq C+\sum_{i=1}^k x_i.$$
If $k\geq 2$, we can choose $i\neq j$ and we get $d_2\geq C-2M = 4M+1-2M > 2M > \delta$, a contradiction. Otherwise $k=1$. But then $\str C = (x_1,d_0,d_2,d_1)$ in this cyclic order. If $d_0=d_2=\delta$ or $d_2=d_1=\delta$, we found a tension. Otherwise we get 
$$2\sum d_i = d_0 + (d_0+d_2) + (d_2+d_1) + d_1 \leq (M+x_1) + (2\delta-1) + (2\delta-1) + (M+x_1),$$
or
$$\sum d_i \leq 2\delta + M + x_1 -1.$$
On the other hand $\sum d_i \geq C+x_1$. And combining these we get $2\delta+M+x_1-1 \geq C+x_1$, or $2\delta+M\geq C+1 = 2\delta+K_2+1 = 2\delta+M+1$, a contradiction.
\end{proof}

In the following we will write $C_x$ for one of the $C_0$ and $C_1$ (to be able to discuss both at once).

\begin{lemma}
If $\str C'$ can be obtained from $\str C$ by a step of the magic completion algorithm, then $\str C'$ is $C_0$- or $C_1$-cycle.
\end{lemma}
\begin{proof}
We know that $\str C$ has a tension, therefore any step of the magic completion algorithm preserves the parity of the perimeter.

If $\delta=5$ and $\str C = (5,5,5,5,5)$, then $\str C' = (5,5,5,2)$ which indeed is a $C_1$-cycle. Otherwise $\str C$ is a $C_x$-cycle.

We know that $\sum x_i < M$, so $x_i\oplus x_j = x_i+x_j$, which preserves the $C_x$-inequality. We also know that $d_i\geq K_2$ for every $i$, so $d_i\oplus d_j\in \{M, C-1-d_i-d_j\}$. Due to tension only $d_i\oplus d_j = C-1-d_i-d_j$ can happen. Suppose without loss of generality $d_i=d_1$ and $d_j=d_2$. Then
$$d_0 + C-1-(d_1\oplus d_2) = \sum d_i \geq C_x + \sum x_i\geq C+\sum x_i$$
and hence
$$d_0 > (d_1\oplus d_2) + \sum x_i,$$
which means that $\str C'$ is a non-metric cycle (which is a $C_x$-cycle).

Finally $d_i\oplus x_j \in \{M, d_i-x_j\}$, the first does not happen due to tension and the second preserves the $C_x$-inequality.
\end{proof}

\begin{lemma}
If $\str C'$ can be obtained from $\str C$ by an inverse step of the magic completion algorithm which expanded edge $p$ to edges $q, r$, then $\str C'$ is a $C_0$- or $C_1$-cycle or $\delta=5$, $\str C=(5,5,5,2)$ and $\str C'=(5,5,5,5,5)$.
\end{lemma}
\begin{proof}
If $p=d_i$, then necessarily $q\oplus r = q-r$, so $\str C'$ is indeed a $C_x$-cycle. If $p=x_i$, then either $q\oplus r = q+r$ (and then again $\str C'$ is a $C_x$-cycle), or $q\oplus r = C-1-q-r$. But then $x_i = K_1-1$ and $q=r=\delta$. It holds that 
$$3\delta \geq \sum d_i \geq C_x+\sum x_i\geq C+K_1-1 = 5K_1 = 5\frac{2\delta-1}{3}$$
with equality only if $C_x = C$, i.e. when $\str C$ is a $C_1$-cycle.

So we have
$$3\delta \geq 5\frac{2\delta-1}{3},$$
or
$$9\delta \geq 10\delta-5,$$
which means $\delta\leq 5$. As in Case~\ref{IIb} it holds that $\delta\geq 5$, we have $\delta = 5$. But then we still need $3\delta \geq C+K_1-1$, or $15\geq 13+3-1$ which is an equality, therefore there had to be equalities in all the estimates, which means that $\str C=(5,5,5,2)$ and thus $\str C'=(5,5,5,5,5)$.
\end{proof}

\section{$C$-cycles with $n\geq 1$ when $C'=C+1$}
In the whole section we let $\str{C}$ be a $C$-cycle with distances $d_0, d_1, \ldots, d_{2n}, x_1, \ldots, x_k$, where $n\geq 1$. We will further use the fact that $C\geq 2\delta+2$ (which is an acceptability condition).

\begin{lemma}\label{lem:structure:C}
It holds that $\sum x_i < d_j$ for every $j$.
\end{lemma}
\begin{proof}
We have $d_j + 2n\delta \geq \sum d_i > n(C-1)+\sum x_i$, and $n(C-1) > 2n\delta$, so $\sum x_i < d_j$.
\end{proof}

\begin{lemma}\label{lem:tension:C}
If $\str C$ has at least 4 vertices, then it has a tension.
\end{lemma}
\begin{proof}
First suppose that $k=0$. Then $n\geq 2$ ($\str C$ has at least four vertices) and without loss of generality $\str C = (d_0, d_1, \ldots, d_{2n})$ in this cyclical order. We identify $d_0=d_{2n+1}$. If for some $0\leq i\leq 2n$ it holds that $C-1-d_i-d_{i+1} < M$, we found a tension. Suppose for a contradiction that $d_i+d_{i+1}\leq C-1-M$ for every $0\leq i\leq 2n$. If we sum these inequalities for all $0\leq i\leq 2n$, we get
$$(2n+1)(C-1-M) \geq 2\sum d_i > 2n(C-1),$$
or
$$C-1 > (2n+1)M.$$
Thus we can again use the $C$-inequality and get
$$(2n+1)\delta\geq \sum d_i > n(C-1) > n(2n+1)M \geq n(2n+1)\frac{\delta}{2},$$
which is clearly absurd as $n\geq 2$.

Now we generalise the previous argument for cases where $k > 0$. If some $d_i, d_j$ are adjacent, then we can assume that $d_i+d_j\leq C-1-M$, as otherwise we have found a tension. If $d_i$ and $x_{j_i}$ are adjacent, then $d_i\leq M+x_{j_i}$, otherwise $d_i\oplus x_{j_i} = d_i-x_{j_i}<M$ and we again have a tension. Suppose that there are $\alpha$ vertices of $\str C$ where some $d_i$ and $d_j$ are adjacent. Then there are $2(2n+1-\alpha)$ vertices, where some $d_i$ and $x_{j_i}$ are adjacent (as there are in total $2(2n+1)$ endpoints of the $d_i$ edges).

If we sum the $d_i+d_j\leq C-1-M$ and $d_i\leq M+x_{j_i}$ inequalities over all endpoints of the $d_i$ edges, we get
$$2\sum d_i \leq \alpha(C-1-M) + 2(2n+1-\alpha)M + \sum x_{j_i}.$$
Clearly $\sum x_{j_i}\leq 2\sum x_i$. We also have a lower bound on $\sum d_i$ from the $C$-inequality, hence
$$\alpha(C-1-M) + 2(2n+1-\alpha)M + 2\sum x_i \geq 2+2n(C-1) + 2\sum x_i,$$
or
$$\alpha(C-1-3M) \geq 2+2n(C-1) - 2(2n+1)M.$$

From the conditions on $M$ we have $M \leq \lfloor \frac{C-1-\delta}{2} \rfloor$. Thus $3M = 2M+M \leq C-1-\delta+M$. If $M=\delta$, then $C\geq 3\delta+1$, but one can easily check that then no cycle is forbidden by the $\str C$-inequality. Hence $M<\delta$ and thus $C-1-3M > 0$.

Now as $n\geq 1$ and $M>0$, we also have $-2(2n+1)M = (-4n-2)M \geq -6nM = 2n(-3M)$. Hence
$$\alpha(C-1-3M) > 2n(C-1-3M),$$
and thus $\alpha > 2n$. But we also know that clearly $\alpha\leq 2n+1$, so $\alpha=2n+1$. But this means that there are no vertices where $d_i$ and $x_j$ meet, hence $k=0$ and the problem was reduced to the previous case.
\end{proof}

\begin{lemma}
If $\str C'$ can be obtained from $\str C$ by a step of the magic completion algorithm, then $\str C'$ is a $C$-cycle.
\end{lemma}
\begin{proof}
We know that $\str C$ has a tension. If the step completed a fork $x_i, x_j$, then $x_i\oplus x_j \in \{x_i+x_j, |x_i-x_j|\}$, in both cases $\str C'$ is a $C$-cycle.

If the step completed a fork $d_i, d_j$ (for convenience we can without loss of generality assume that it was $d_{2n-1}, d_{2n}$), then $d_{2n-1}\oplus d_{2n} \in \{d_{2n-1}+ d_{2n}, |d_{2n-1}- d_{2n}|, C-1-d_{2n-1}- d_{2n}\}$. Clearly $d_{2n-1}\oplus d_{2n} = d_{2n-1}+ d_{2n} < M$ is absurd, as then $(2n-1)\delta + M > \sum d_i \geq 1+n(C-1)\geq 1+2n\delta + n$, a contradiction.

If $d_{2n-1}\oplus d_{2n} = |d_{2n-1} - d_{2n}|$, say $d_{2n-1}-d_{2n}$, then
$$\sum_{i=0}^{2n-2} d_i \geq 1+(n-1)(C-1) + \sum x_i + C-1-d_{2n-1}-d_{2n}.$$
But
$$C-1-d_{2n-1}-d_{2n} \geq d_{2n-1} - d_{2n},$$
as
$$C-1\geq 2\delta \geq 2d_{2n-1}.$$
So
$$\sum_{i=0}^{2n-2} d_i \geq 1+(n-1)(C-1) + \sum x_i + (d_{2n-1}\oplus d_{2n})$$
and thus $\str C'$ is a $C$-cycle.

Otherwise $d_{2n-1}\oplus d_{2n} = C-1-d_{2n-1}- d_{2n}$, but then 
$$\sum_{i=0}^{2n-2} d_i \geq 1+(n-1)(C-1) + \sum x_i + C-1-d_{2n-1}-d_{2n}$$
$$ = 1+(n-1)(C-1) + \sum x_i + (d_{2n-1}\oplus d_{2n}).$$

Finally $d_i\oplus x_j \in \{d_i+x_j, d_i-x_j, C-1-d_i-x_j\}$ (as $d_i\geq x_j$ we have $|d_i-x_j| = d_i-x_j$). If $d_i\oplus x_j \in \{d_i+x_j, d_i-x_j\}$, then $\str C'$ indeed is a $C$-cycle. Suppose now that $d_i\oplus x_j = C-1-d_i-x_j$. If the second neighbour of $d_i$ is some $d_l$, then either $d_l > x_j$ and then $t(C-1-d_i-d_l) < t(C-1-d_i-x_j)$ which is a contradiction, or $d_l=x_j$, we can simply swap the roles of $d_l$ and $x_j$ and thus have reduced it to the previous paragraph. Otherwise the second neighbour of $d_i$ is some $x_l$.

We know that $x_j+x_l\leq \sum x_i < \delta$. And then clearly $\delta+x_j+x_l < C-1$. But this means that $\delta-d_i+x_l < C-1-d_i-x_j$. This implies that $d_i-x_l>M$ (because, by the assumption $C-1-d_i-x_j < M$ and $M\geq \lceil\frac\delta 2\rceil$) and also that $2(\delta-d_i+x_l) < 2(C-1-d_i-x_j) + 1$. Together, this means $t(d_i-x_l) < t(C-1-d_i-x_j)$, which is a contradiction.
\end{proof}

\begin{lemma}
If $\str C'$ can be obtained from $\str C$ by an inverse step of the magic completion algorithm which expanded edge $p$ to edges $q, r$, then $\str C'$ is a $C$-cycle.
\end{lemma}
\begin{proof}
If $p=x_i$ and $q\oplus r = q+r$, then $\str C'$ is trivially a $C$-cycle. If $p=x_i$ and $q\oplus r = C-1-q-r$, then
$$q+r+\sum d_i \geq 1+(n+1)(C-1)+\sum_{j\neq i} x_j,$$
so $\str C'$ is again a $C$-cycle. Now suppose that $p=x_i$ and $q\oplus r = q-r$. This means that $x_i > M$. We get for every $j\neq l$ that
$$(2n-1)\delta + d_j+d_l\geq \sum d_i \geq 1+n(C-1) + x_i \geq 1 + 2(n-1)\delta + C-1 + x_i,$$
that is
$$d_j+d_l\geq C-\delta+x_i.$$
In particular it holds for every neighbouring $d_j, d_l$. Then $C-1-d_j-d_l \leq \delta-1-x_i < M$, hence $t(C-1-d_j-d_l) \leq 2\delta-2x_i-1$. But $t(x_i) = 2\delta-2x_i$, hence $t(C-1-d_j-d_l)<t(x_i)$, a contradiction. This means that there can be no neighbouring $d_j, d_l$.

Similarly, for every neighbouring $d_j, x_{l_j}$ we get $d_j-x_{l_j} \leq x_i$ (otherwise we again get a contradiction with time). If some $d_j$ is neighbouring with $q$, let it without loss of generality be $d_0$ and if some $d_j$ is adjacent to $r$, let it without loss of generality be $d_1$. Now we can sum all the valid inequalities $d_j\leq x_i+x_{l_j}$ and get
$$2\sum d_j \leq 4nx_i + 2\sum_{j\neq i} x_j + d_0 + d_1 = (4n-2)x_i + 2\sum x_j + d_0+d_1,$$
because there are $2(2n+1)$ endpoints of the $d_j$ edges and at most two of them are not counted.
Thus
$$(4n-2)x_i+2\delta + 2\sum x_j \geq 2\sum d_j \geq 2+2n(C-1)+2\sum x_j,$$
hence
$$(4n-2)x_i+2\delta \geq 2+2n(C-1)> 2+4n\delta,$$
but this means
$$(4n-2)x_i > (4n-2)\delta,$$
which is clearly absurd.

\medskip

The other possibility is without loss of generality $p=d_0$. If $q\oplus r = q-r$, then $\str C'$ is again a $C$-cycle. Now suppose that $q\oplus r \in \{C-1-q-r, q+r\}$. We know that $M>d_0 > \sum x_i$. If there are two $x_i$'s adjacent, their $\oplus$-sum is smaller than $d_0$, a contradiction with time. Hence no two $x_i$'s are adjacent.

By rearranging the $C$-inequality, we get for every $j\neq 0$ and every $1\leq l\leq k$ the following inequality
$$d_j-x_l \geq d_j-\sum x_i \geq 1+n(C-1) - \sum_{m\notin \{0, j\}} d_m - d_0,$$
but $n(C-1) \geq 2n\delta + n$ and $\sum_{m\notin \{0, j\}} d_m \leq (2n-1)\delta$, hence
$$d_j-x_l\geq 1+\delta+n-d_0,$$
so $2(\delta - (d_j-x_l)) \leq 2(d_0-1-n) < 2d_0+1$. If $d_j-x_l>M$ and $d_j$ and $x_l$ are adjacent, we would get $t(d_j-x_l) < t(d_0)$, a contradiction with time. Hence for every $j\neq 0$ we get that if $d_j$ and $x_l$ are adjacent, it holds that $d_j-x_l\leq M$, or in other words $x_l\geq d_j-M$.

There are at most two vertices in which some $x_i$ is adjacent to $d_0$. For every vertex where some $x_i$ is adjacent to $d_{j_i}$ with $j_i\neq 0$ we have $x_i\geq d_{j_i}-M$. We know that $d_0$ is adjacent to $\ell\in\{0,1,2\}$ edges labelled by some $x_i$. Thus we can bound
$$2\sum_{i=1}^k x_i \geq \sum_{j_i\neq 0} d_{j_i} - (2k-\ell)M,$$
where $\sum_{j_i\neq 0} d_{j_i}$ goes over all the neighbours of all $x_i$'s, that is, every $d_i$ occurs in the sum at most two times and there are in total $2k-\ell$ summands.
By rearranging the $C$-inequality and multiplying it by 2 we know that
$$2d_0 + 2\sum_{i=1}^{2n}d_i > 2n(C-1)+2\sum x_i.$$
Combining the last two inequalities with the trivial bound $d_i\leq \delta$ and the assumption $d_0<M$, we get
$$2M+(4n-2k+\ell)\delta+(2k-\ell)M > 2n(C-1).$$
If we multiply this inequality by $2$ and bound $\frac{C-1-\delta}{2}\geq M$, we get
$$(8n-4k+2\ell)\delta +(2k-\ell+2)(C-1) - (2k-\ell+2)\delta > 4n(C-1),$$
or
$$(8n-6k+3\ell-2)\delta > (4n-2k+\ell-2)(C-1).$$
Now we can bound $C-1>2\delta$ and divide the inequality by $\delta$ to get
$$8n-6k+3\ell-2 > 8n-4k+2\ell-4,$$
or
$$2 > 2k-\ell.$$
We know that $k\geq 0$, $\ell\in\{0,1,2\}$ and $k\geq \ell$. It then follows that $(k,\ell)\in \{(0,0),(1,1)\}$.

This means that without loss of generality either $\str C = (d_0, d_1, \ldots, d_{2n})$, or $\str C = (x_1, d_0, d_1, \ldots, d_{2n})$. In both cases it must hold that $C-1-d_i-d_{i+1} \geq d_0$ for every $1\leq i\leq 2n-1$, as otherwise it would contradict time. However, by the $C$-inequality, $\sum d_i > n(C-1)$, so $d_0 > \sum_{i=1}^n \left(C-1-d_{2i-1}-d_{2i}\right) \geq nd_0$, which is a contradiction. Thus $q\oplus r \notin \{C-1-q-r, q+r\}$ and we are done.
\end{proof}

\section{($1,\delta$)-graphs}
\label{sec:1delta}
We conclude with a short note about ($1,\delta$)-graphs associated to the
metrically homogeneous graphs of diameter $\delta$. Our discussion is based
on the following easy observation.
\begin{observation}
\label{obs:cycles}
In every metric space associated to a metrically homogeneous graph of finite diameter $\delta$, every pair of vertices in distance $1\leq d\leq \delta$ is
connected by 
\begin{enumerate}
\item \label{geodesic} a path consisting of $d$ edges of length 1 (a geodesic path); and
\item \label{antigeodesic} a path consisting of one edge of length $\delta$ and $\delta-d$ edges of length 1 (an anti-geodesic path).
\end{enumerate}
\end{observation}
\begin{proof}
Let $u,v$ be a pair of vertices in distance $d$.
Part~(\ref{geodesic}) follows from the definition of the associated metric space. 

To see (\ref{antigeodesic}), consider a vertex $u'$ in distance $\delta$
from $u$ (such a vertex exists because the metric space is homogeneous and has
diameter $\delta$).
Consequently, there is a path consisting of $\delta$ edges of length 1 connecting $u$ and $u'$.
This path contains a vertex $u''$ which is in distance $d$ from $u$.
Hence, the triangle $u',u,u''$ has distances $\delta$, $d$ and $\delta-d$.
Homogeneity implies that there is also a vertex in distance $\delta$ from $u$ and 
$\delta-d$ from $v$.
\end{proof}
This suggests a ``reverse approach'' to the study of metrically
homogeneous graphs with strong amalgamation and finite diameter $\delta$.
Rather than specifying constrains on the metric space, one can give
constraints in the form of forbidden cycles having only edges
of length 1 and $\delta$. All other distances are then uniquely determined
by means of Observation~\ref{obs:cycles}. In this setting, it suffices to only consider \emph{$(1,\delta)$-graphs},
that is, edge-labelled graphs created from the associated metric space by only
keeping edges of length $1$ and $\delta$: Every distance then corresponds to a unique orbit of 2-tuples.
The associated metric space can be then seen as the unique
\emph{homogenization} in the sense of
\cite{Covington1990,Hubicka2009,Hubicka2013,Hubicka2016}: every distance is uniquely
associated with an orbit of 2-tuples of the automorphism group of the $(1,\delta)$-graph.

It is easy to re-interpret Definition~\ref{defn:forbcycles} for $(1,\delta)$-edge-labelled cycles:

\begin{definition}[Forbidden ($1,\delta$)-cycles] \label{defn:1Dforbcycles}
Let $(\delta,\allowbreak K_1,\allowbreak K_2,\allowbreak C_0,\allowbreak C_1)$ be an admissible sequence of parameters. 
Denote by $\mathcal C_{i,j}$ the family of all $(1,\delta)$-cycles consisting
of $i$ edges of length $\delta$ and $j$ edges of length $1$.

The following are the building blocks of $\mathcal F_{1,\delta}$:
\begin{description}
\item[$K_1$-cycles] $\mathcal C_{0,j}$, $j$ is odd and $j< 2K_1$.
\item[Non-metric cycles] $\mathcal C_{1,j}$ such that $j<\delta$.
\item[$K_2$-cycles] $\mathcal C_{i,j}$ such that $i\geq 2$ is even, $j$ is odd, and $$2j<2C-4K_2-2-(C-1-2\delta)i.$$
\item[$C$-cycles] If $C=C'+1$ then all $\mathcal C_{i,j}$ such that $i\geq 3$ is odd and $$2j< C-1 - (C-1-2\delta)i.$$
\item[$C_0$-cycles] If $C>C'+1$ then all $\mathcal C_{i,j}$ such that $i=3$, $\delta+j$ is even and $$2j< C_0-1 - (C_0-1-2\delta)i.$$
\item[$C_1$-cycles] If $C>C'+1$ then all $\mathcal C_{i,j}$ such that $i=3$, $\delta+j$ is odd and $$2j< C_1-1 - (C_1-1-2\delta)i.$$
\item[The $C_1^5$-cycle] If $C>C'+1$, $\delta=5$ and the parameters come from Case~\ref{IIb}, then $\mathcal C_{5,0}=\{(5,5,5,5,5)\}$ is also forbidden.
\end{description}
\end{definition}
It follows that cycles with 0 edges of length $\delta$ are constrained only by
$K_1$, cycles with 1 edge of length $\delta$  are constrained only by
$\delta$ (non-metric cycles), cycles with $2i$ edges of length $\delta$, $i>1$ are
constrained only by $K_2$ and cycles with $2i+1$ edges of length $\delta$,
$i>1$ are constrained by $C$.
Also observe that $C-1-2\delta$ is the distance used to complete fork $\delta$-$\delta$
by the magic completion algorithm.

\begin{remark*}
The name ``$C_1^5$-cycle'' was not chosen haphazardly. The reason why the $C$-cycles and the $K_2$-cycles have the ``$n(C-1)$'' part in the respective inequalities is that in these cases the $d^C$-fork is used quite heavily. On the other hand, it turns out (cf. Fact~\ref{fact:oplus} and Section~\ref{sec:c_1iib}) that when $C'>C+1$ and the parameters are admissible, the inverse steps of the magic completion algorithm almost never use the $d^C$-fork (the exception being non-metric cycles and the very special case which produces the $C_1^5$-cycle). Although conceptually, the $C_1^5$-cycle really \textit{is} a $C_1$-cycle, for the purposes of this paper it was more convenient not to define the $C_0$- and $C_1$-cycles in full generality and treat the $C_1^5$-cycle as a special case.
\end{remark*}

\begin{table}

\begin{tabular}{r c  c  c  c  c  c |}
	                               & $0$                   & $1$         & $2$         & $3$         & $4$ & \multicolumn{1}{c}{$5$} \\
	\cline{5-7}
	$0\delta$                      & \multicolumn{3}{c|}{} & \small $K_1$       &             & \small $K_1$                                       \\\cline{4-4}
	$1\delta$                      & \multicolumn{2}{c|}{} & \small $\delta$ & \small $\delta$ & \small $\delta$ &                               \\\cline{3-3}
	$2\delta$                      & \multicolumn{1}{c|}{} & \small $K_2$       &             & \small $K_2$       &     &                         \\\cline{2-2}
	\multicolumn{1}{r|}{$3\delta$} & \small $C_1$                 &             & \small $C_1$       &             &     &                         \\
	\multicolumn{1}{r|}{$4\delta$} &                       & \small $K_2$       &             &             &     &                         \\
	\multicolumn{1}{r|}{$5\delta$} & \small $C_1^5$                 &             &             &             &     &                         \\\cline{2-7}\\\end{tabular}
\caption{Forbidden $(1,\delta)$-cycles for $\delta=5, K_1=3, K_2=3, C_0=16, C_1=13$}
\label{table:1}
\end{table}

\begin{table}

\begin{tabular}{r c  c  c  c |}
	                               & $0$                   & $1$         & $2$         & \multicolumn{1}{c}{$3$} \\
	\cline{5-5}
	$0\delta$                      & \multicolumn{3}{c|}{} &                                                     \\\cline{4-4}
	$1\delta$                      & \multicolumn{2}{c|}{} & \small $\delta$ & \small $\delta$                           \\\cline{3-3}
	$2\delta$                      & \multicolumn{1}{c|}{} & \small $K_2$       &             &                         \\\cline{2-2}
	\multicolumn{1}{r|}{$3\delta$} & \small $C$                     &\small $C$           &             &                         \\\cline{2-5}\\\end{tabular}
\begin{tabular}{r c  c  c  c |}
	                               & $0$                   & $1$         & $2$         & \multicolumn{1}{c}{$3$} \\
	\cline{5-5}
	$0\delta$                      & \multicolumn{3}{c|}{} & \small $K_1$                                                    \\\cline{4-4}
	$1\delta$                      & \multicolumn{2}{c|}{} &\small  $\delta$ & \small $\delta$                           \\\cline{3-3}
	$2\delta$                      & \multicolumn{1}{c|}{} & \small $K_2$       &             &                         \\\cline{2-2}
	\multicolumn{1}{r|}{$3\delta$} &                       & \small $C_1$       &             &                         \\\cline{2-5}\\\end{tabular}
\caption{Forbidden $(1,\delta)$-cycles for $\delta=4, K_1=1, K_2=3, C_0=14, C_1=11$
and $\delta=4, K_1=2, K_2=3, C_0=12, C_1=11$.}
\label{table:2}
\end{table}

\begin{table}
\begin{tabular}{r c  c  c  c |}
	                               & $0$                   & $1$         & $2$         & \multicolumn{1}{c}{$3$} \\
	\cline{5-5}
	$0\delta$                      & \multicolumn{3}{c|}{} &                                                \\\cline{4-4}
	$1\delta$                      & \multicolumn{2}{c|}{} & \small $\delta$ &\small  $\delta$                           \\\cline{3-3}
	$2\delta$                      & \multicolumn{1}{c|}{} & \small $K_2$       &             &                         \\\cline{2-2}
	\multicolumn{1}{r|}{$3\delta$} &                       & \small $C_1$       &             &                         \\\cline{2-5}\\\end{tabular}
\caption{Forbidden $(1,\delta)$-cycles for $\delta=4, K_1=1, K_2=3, C_0=14, C_1=11$.}
\label{table:3}
\end{table}

The distribution of individual constrains can be visualised as
shown in Table~\ref{table:1}. Here,
symbol with the coordinates $i\delta,j$ specifies that cycles with $i$
edges of length $\delta$ and $j$ edges of length 1 are forbidden by the
corresponding bound ($\delta$ denotes non-metric cycles).
Observe that whenever the cycles $\mathcal C_{i,j}$ are forbidden than also the cycles
$\mathcal C_{i-2, j}$ and $\mathcal C_{i, j-2}$ are forbidden whenever they make sense.
Moreover one can not forbid cycles where both the number of edges of length 1 and the number of edges of length $\delta$ are even.
This explains why the forbidden cycles ``form an upper left triangle''
and why there is at most one different type of bound for every even row/column and at most two bounds
for every odd row/column (the cycle $(5,5,5,5,5)$ can be in fact understood as a $C_1$-cycle with $n=2$, or, in this $(1,\delta)$-formalism, $C_1$-cycle with $i=5$, for our purposes it was, however, more convenient to treat it as a special case).

Several properties of the metrically homogeneous graphs can be seen from this
table.  For example,~\cite{coulson2018twists} identifies pairs of metrically
homogeneous graphs whose automorphism groups are isomorphic (and thus the
associated metric spaces are the same up to a non-trivial permutation of distances --- a \emph{twisted isomorphism}).
These are pairs of metrically homogeneous graphs such that the table for one has non-empty cells exactly where the transposition of the other has non-empty cells.
One such pair is shown in Table~\ref{table:2}. If the table is symmetric across
the diagonal, then the metrically homogeneous graph has a twisted automorphism to
itself as shown in Table~\ref{table:3}.
This covers all the exceptional
cases identified in~\cite{coulson2018twists}. The regular case corresponds
to the situation where either edges of length 1 or edges of length $\delta$
are not necessary to preserve the structure of the metrically homogeneous graph, that is, when it is already described by a $1$-graph or a $\delta$-graph in the sense of this chapter.

We shall also remark that in the sense of Cherlin, Shelah, and Shi~\cite{Cherlin1999,Cherlin2001}, the
metric spaces associated to the metrically homogeneous graphs are 
the existentially complete universal structures for the classes of
countable $(1,\delta)$-graphs omitting homomorphic images of the
given set of forbidden $(1,\delta)$-cycles.
This connection and more consequences will be explored in greater detail
elsewhere.

\bibliography{ramsey.bib}

\end{document}